\documentclass{article}
\makeatletter

\usepackage{amsfonts}
\usepackage{amsmath}
\usepackage{color}
\usepackage{bm}
\usepackage{amsthm}
\usepackage{amssymb}
\usepackage{graphics}
\usepackage{pictex}

\newtheorem{theorem}{Theorem}[section]
\newtheorem{proposition}{Proposition}[section]
\newtheorem{definition}{Definition}[section]

\newtheorem{lemma}{Lemma}[section]
\newtheorem{remark}{Remark}[section]
\begin{document}

\title{Existence and convergence of solutions for nonlinear biharmonic equations on graphs}
\date{ }

\author{Xiaoli Han$^{1,a}$\ \ \ Mengqiu Shao$^{1,b}$\ \ \ Liang Zhao$^{2,c}$\\
\it\small  $^{1}$  Department of Mathematical Sciences, Tsinghua University, Beijing 100084, China\\
\it\small  $^{2}$  School of Mathematical Sciences, Laboratory of Mathematics and Complex Systems of MOE, \\
\it\small Beijing Normal University, Beijing 100875, China\\
\it\small E-mail: $^{a}$ hanxiaoli@mail.tsinghua.edu.cn, $^{b}$shaomq17@mails.tsinghua.edu.cn,  ${^c}$ liangzhao@bnu.edu.cn
}
\maketitle

\begin{abstract}
In this paper, we first prove some propositions of Sobolev spaces defined on a locally finite graph $G=(V,E)$, which are fundamental when dealing with equations on graphs under the variational framework. Then we consider a nonlinear biharmonic equation
$$
\Delta^{2} u -\Delta u+(\lambda a+1)u= |u|^{p-2}u
$$
on $G=(V,E)$. Under some suitable assumptions, we prove that for any $\lambda>1$ and $p>2$, the equation admits a ground state solution $u_{\lambda}$. Moreover, we prove that as $\lambda\rightarrow +\infty$, the solutions $u_{\lambda}$ converge to a solution of the equation
$$
\left\{\aligned &\Delta^{2}u -\Delta u+u = |u|^{p-2}u, &\hbox{in}\ \ \Omega,\\
&u=0, &\hbox{on}\ \ \partial\Omega,\endaligned\right.
$$
where $\Omega=\{x\in V: a(x)=0\}$ is the potential well and $\partial\Omega$ denotes the the boundary of $\Omega$.
\end{abstract}


\section{Introduction}

Graph is a natural structure for problems with different backgrounds, such as image processing \cite{Lezoray}, neural network \cite{Bose}, social network \cite{Arnaboldi}, etc. To do analysis works, it is necessary to study partial differential equations on graphs and this subject has attracted much attention recently. For example, several fundamental aspects of heat equations on graphs, such as heat kernel \cite{Horn,Woj}, existence and uniqueness\cite{Huang,LinWu2,LiuChenYu} are investigated by different authors.

In this paper, we use methods in functional analysis to study existence of solutions for fourth order nonlinear elliptic equations on graphs. Our ideas are inspired by works of Grigor'yan, Lin and Yang \cite{GLY1,GLY2,GLY3} where they considered several second order nonlinear elliptic equations on graphs. For example, when domains are finite graphs, they proved existence of solutions for the Kazdan-Warner equation \cite{GLY1} and the Yamabe type equation \cite{GLY2}. Later, results in \cite{GLY1,GLY2} were generalized by Ge and Jiang \cite{GeJiang1,GeJiang2} for infinite graphs and Keller and Schwarz \cite{KellerSchwarz} also studied the Kazdan-Warner equation on canonically compact graphs.

Moreover, we also study the asymptotic behavior of solutions to our equation and find that their limit is restricted on the potential well which is a finite graph. On graphs, this kind of results was first proved by Zhang and Zhao. In \cite{ZhangZhao}, they studied the following second order equation,
\begin{equation}\label{loworder}
-\Delta u+(\lambda a+1)u=|u|^{p-1}u
\end{equation}
on a locally finite graph $G=(V,E)$, where $a(x)$ is a potential function defined on $V$. If the potential well $\Omega=\{x\in V: a(x)=0\}$ is a non-empty, connected and bounded domain in $V$, their results said that, as $\lambda\rightarrow +\infty$, the ground state solutions $u_\lambda$ of \eqref{loworder} converge to a ground state solution $u_0$ of the corresponding Dirichlet equation,
\begin{align*}\label{lowdirichlet}
\begin{cases}
-\Delta u+u=|u|^{p-1}u \  &\text{in}\  \Omega;\\
u=0,\  &\text{on} \  \partial \Omega.
\end{cases}
\end{align*}

The equation \eqref{loworder} is a Sch\"{o}dinger type equation and when the domain is a subset of the Euclidean space, it has been extensively studied during the past several decades. The readers can refer to \cite{Alves,BartschWang,BrezisNirenberg,Cao,Rabinowitz,WangZhou,Yang} and the references therein. Besides the second order equations, when we consider problems with certain physical backgrounds, such as travelling waves in a suspension bridge \cite{Mckenna} and the static deflection of an elastic plate \cite{Abrahams}, there arises the higher order version of \eqref{loworder} with a biharmonic operator and it is also well studied. For example, in \cite{LiuChen}, Liu and Chen studied the multiplicity of solutions of a biharmonic Sch\"{o}dinger equation with critical growth. In \cite{NiuTangWang}, Niu, Tang and Wang studied the asymptotic behavior of ground state solutions for a nonlinear biharmonic equation on $\mathbb{R}^N$. For more results related to nonlinear biharmonic equations on the Euclidean space, one can refer to \cite{GuoHuangZhou,Sani,WS,ZhangLouJiShao,ZhaoZhang} and the references therein.

Noticing the above works, the main purpose of this paper is to investigate whether the asymptotic result in \cite{ZhangZhao} still holds on graphs for nonlinear biharmonic equations. To describe our problems and results, we first introduce some concepts and assumptions. let $G = (V, E)$ be a graph, where $V$ denotes the set of vertices and $E$ denotes the set of edges. In this paper, we always assume that $G$ satisfies the following conditions.

\noindent$(G_{1})$ {\it Locally finite. For any $x \in V$, there are only finite $y\in V$ such that $ xy\in E$ and $G$ is called a locally finite graph. }

\noindent$(G_{2})$ {\it Connected. $G$ is called connected if any two vertices $x$ and $y$ can be connected via finite edges.}

\noindent$(G_{3})$ {\it Uniformly positive measure. For a measure $\mu:V\rightarrow \mathbb{R}^{+}$ defined on $V$, we assume that there exists a constant $\mu_{\min}>0$ such that $\mu(x)\geq\mu_{\min}$ for all $x\in V$.}

\noindent$(G_{4})$ {\it Symmetric. For an edge $xy\in E$, we assume it has a positive weight $\omega_{xy}$ and it is symmetric, namely $\omega_{xy}=\omega_{yx}$. Furthermore, we assume that for any $x\in V$, $\underset{y\sim x}\sum\omega_{xy}<C$, where $C$ is a universal constant. Here and throughout this paper, $y\sim x$ stands for any vertex $y$ connected with $x$ by an edge $xy\in E$.}

The distance $d(x,y)$ of two vertices $x,y\in V$ is defined by the minimal number
of edges which connect these two vertices. For a subset $\Omega$ of $V$, if the distance $d(x,y)$ is uniformly bounded from above for any $x,y\in\Omega$, we call $\Omega$ a bounded domain in $V$. We shall remark that a bounded domain of a locally finite graph can contain only finite vertices. The boundary of $\Omega$ in $V$ is defined by
$$\partial\Omega:=\{y\in V, y\notin\Omega: \exists x\in\Omega\ \ \hbox{such that} \ \ xy\in E\}$$
and the interior of $\Omega$ is denoted by $\Omega^{\circ}$. Obviously, we have that $\Omega^{\circ}=\Omega$.

For any function $u:V\rightarrow \mathbb{R}$ and $x\in V$, the $\mu$-Laplacian (or Laplacian for short) of $u$ at $x$ is defined by
$$\Delta u(x):=\frac{1}{\mu(x)}\underset{y\sim x}\sum \omega_{xy}(u(y)-u(x)).$$
The  gradient form $\Gamma(u,v)$ of two functions $u$ and  $v$ at $x\in V$ is
$$\Gamma(u,v)(x):=\frac{1}{2\mu(x)}\underset{y\sim x}\sum\omega_{xy}(u(y)-u(x))(v(y)-v(x)).$$
For brevity, we use $\Gamma(u)$ for $\Gamma(u,u)$ and sometimes we use $\nabla u\nabla v$ instead of $\Gamma(u,v)$.
The length of $\Gamma(u)$ at $x\in V$ is denoted by
$$|\nabla u|(x)=\sqrt{\Gamma(u)(x)}=\left(\frac{1}{2\mu(x)}\underset{y\sim x}\sum\omega_{xy}(u(y)-u(x))^{2}\right)^{\frac{1}{2}}.$$
For any function $u: V\rightarrow\mathbb{R}$, an integral of $u$ over $V$ is defined by
$$\int_{V}ud\mu=\underset{x\in V}\sum\mu (x)u(x).$$
The biharmonic operator of $u: V\rightarrow\mathbb{R}$, namely $\Delta^{2} u$, is defined in the distributional sense by
$$\int_{V}(\Delta^{2}u)\phi d\mu=\int_{V}\Delta u \Delta\phi d\mu,\ \ \forall\phi\in C_{c}(V),$$
where $C_{c}(V):=\{u:V\rightarrow\mathbb{R}:\{x\in V: u(x)\neq 0\} \ \ \hbox{is of  finite cardinality}\}.$

In this paper, we focus on the following nonlinear biharmonic equation
\begin{equation}\label{equation}
 \Delta^{2} u -\Delta u+(\lambda a+1)u= |u|^{p-2}u
\end{equation}
on a graph $G=(V,E)$ satisfying $(G_1)-(G_4)$. Here $\lambda>1$ and $p>2$ are constants and $a(x):V\rightarrow \mathbb{R}$ is a potential satisfying:

\noindent$(A_{1})$ {\it $a(x)\geq 0$ and the potential well $\Omega=\{x\in V: a(x)=0\}$ is a non-empty, connected and bounded domain in $V$.}

\noindent$(A_{2})$ {\it There exists a vertex $x_{0}\in V$ such that $a(x)\rightarrow +\infty$ as $d(x,x_{0})\rightarrow +\infty$.}

Let $W^{2,2}(V)$ be the completion of $C_{c}(V)$ under the norm
$$\|u\|_{W^{2,2}(V)}=\left(\int_{V}(|\Delta u|^{2}+|\nabla u|^{2}+u^{2})d\mu\right)^{\frac{1}{2}}.$$
Clearly, $W^{2,2}(V)$ is a Hilbert space with the inner product
$$(u,v)_{W^{2,2}(V)}=\int_{V}(\Delta u\Delta v+\nabla u\nabla v+uv)d\mu,\ \ \forall u,v\in W^{2,2}(V).$$
To study the problem \eqref{equation}, it is natural to consider a function space
$$E_{\lambda}:=\{u\in W^{2,2}(V):\int_{V} \lambda a u^{2}d\mu<+\infty\}$$
with the norm
$$\|u\|_{E_{\lambda}}=\left(\int_{V}(|\Delta u|^{2}+|\nabla u|^{2}+(\lambda a+1)u^{2})d\mu\right)^{\frac{1}{2}}.$$
The space $E_{\lambda}$ is also a Hilbert space with its  inner product
$$(u,v)_{E_{\lambda}}=\int_{V}(\Delta u\Delta v+\nabla u\nabla v+(\lambda a+1)uv)d\mu,\ \ \forall u,v\in E_{\lambda}.$$

The functional related to \eqref{equation} is
\begin{equation}\label{functional}
J_{\lambda}(u)=\frac{1}{2}\int_{V}(|\Delta u|^{2}+|\nabla u|^{2}+(\lambda a+1)u^{2})d\mu-\frac{1}{p}\int_{V}|u|^{p}d\mu.
\end{equation}
We can easily verify that $J_{\lambda}\in C^{1}(E_{\lambda},\mathbb{R})$
and
\begin{equation}\label{frechet}
J^{'}_{\lambda}(u)v=\int_{V}(\Delta u\Delta v+\nabla u\nabla v+(\lambda a+1)uv)d\mu-\int_{V}|u|^{p-2}uvd\mu,\ \ \forall v\in E_{\lambda}.
\end{equation}
The Nehari manifold related to \eqref{equation} is defined as
$$\mathcal{N}_{\lambda}:=\{u\in E_{\lambda}\setminus\{0\}:J'_{\lambda}(u)u=0\}.$$
Let $m_{\lambda}$ be
\begin{equation}\label{minimum}
m_{\lambda}:=\underset{u\in \mathcal N_{\lambda}}\inf J_{\lambda}(u).
\end{equation}
If $m_{\lambda}$ can be achieved by some function $u_{\lambda}\in\mathcal{N}_{\lambda}$,  $u_{\lambda}$  have the the least energy among all functions belong to the Nehari manifold and in fact, it is a critical point of the functional $J_{\lambda}$. We call $u_{\lambda}$ a ground state solution of \eqref{equation}. Our first result is the following theorem:

\begin{theorem}\label{existence1}
Let $G=(V,E)$ be a graph satisfies $(G_1)-(G_4)$. Assume $a(x):V\rightarrow [0,+\infty)$ is a function satisfying $(A_{1})$ and $(A_{2})$. Then for any positive constants $\lambda>1$ and $p>2$, there exists a ground state solution $u_{\lambda}$ of the equation \eqref{equation}.
\end{theorem}

For the asymptotic behavior of $u_{\lambda}$ as $\lambda\rightarrow+\infty$, we introduce the limit problem which is defined on the potential well $\Omega$:
\begin{align}\label{dirichlet}
\begin{cases}
\Delta^{2} u-\Delta u+u=|u|^{p-2}u \  &\text{in}\  \Omega;\\
u=0,\  &\text{on} \  \partial \Omega.
\end{cases}
\end{align}

Define the space $W^{2,2}(\Omega)$ as a set of all functions $u:V\rightarrow\mathbb{R}$ under the norm
$$\|u\|_{W^{2,2}(\Omega)}=\left(\int_{\Omega\cup\partial\Omega}(|\Delta u|^{2}+|\nabla u|^{2})d\mu+\int_{\Omega}u^{2}d\mu\right)^{\frac{1}{2}}.$$
It is suitable to study \eqref{dirichlet} in the space $H(\Omega):=W^{2,2}(\Omega)\cap W^{1,2}_{0}(\Omega)$ , where $W^{1,2}_{0}(\Omega)$ is the completion of $C_{c}(\Omega)$ under the norm
$$\|u\|_{W^{1,2}_{0}(\Omega)}=\left(\int_{\Omega\cup\partial\Omega}|\nabla u|^{2}d\mu+\int_{\Omega}u^{2}d\mu\right)^{\frac{1}{2}},$$
where $C_{c}(\Omega)$ denotes the set of all functions $u: \Omega\rightarrow\mathbb{R}$ satisfying $supp\  u\subset\Omega$ and $u=0$ on $\partial\Omega$.
The space $H(\Omega)$ endowed with the inner product
$$(u,v)_{H(\Omega)}=\int_{\Omega\cup\partial\Omega}(\Delta u\Delta v+\nabla u\nabla v)d\mu+\int_{\Omega}uvd\mu$$
is a Hilbert space.
The functional related to \eqref{dirichlet} is
\begin{equation}\label{functional_d}
J_{\Omega}(u)=\frac{1}{2}\int_{\Omega\cup\partial\Omega}(|\Delta u|^{2}+|\nabla u|^{2})d\mu+\frac{1}{2}\int_{\Omega}u^{2}d\mu-\frac{1}{p}\int_{\Omega}|u|^{p}d\mu.
\end{equation}
The corresponding Nehari manifold is
$$\mathcal{N}_{\Omega}:=\{u\in H(\Omega)\setminus\{0\}:J^{'}_{\Omega}(u)u=0\}.$$
And
$$m_{\Omega}:=\underset{u\in \mathcal{N}_{\Omega}}\inf J_{\Omega}(u).$$

Similar to Theorem 1.1, the equation \eqref{dirichlet} also has a ground state solution.

\begin{theorem}\label{existence2}
Let $G=(V,E)$ be a graph satisfies $(G_1)-(G_4)$ and $\Omega$ be a non-empty, connected and bounded domain in $V$. Then for any $p>2$, the equation \eqref{dirichlet} has a ground state solution $u_{0}\in H(\Omega)$.
\end{theorem}

Finally, as $\lambda\rightarrow +\infty$, we prove that the solutions $u_{\lambda}$ converge to a solution of \eqref{dirichlet}. More precisely, we have

\begin{theorem}\label{convergence}
Under the same assumptions as in Theorem 1.1, we have that, for any sequence $\lambda_{k}\rightarrow \infty$, up to a subsequence, the corresponding ground state solutions $u_{\lambda_{k}}$ of \eqref{equation} converge in $W^{2,2}(V)$ to a ground state solution of \eqref{dirichlet}.
\end{theorem}

As far as we know, there is no such results on forth order equations defined on locally finite graphs. Our works generalize the results in \cite{ZhangZhao} to higher order equations but the proofs are more complicated than those in \cite{ZhangZhao}. Furthermore, our equations and proofs on a graph are not the same as those in the Euclidean space. For example, since a graph is a discrete object, there is no derivative on the boundary of a bounded domain in $G$ and the boundary condition of \eqref{dirichlet} is different from those in the Euclidean space. Usually in the Euclidean case, the existence of solutions for the equation \eqref{dirichlet} is proved by minimizing the corresponding functional or by the Mountain Pass theorem. Instead of these methods, we get the existence of a ground state solution for the equation \eqref{dirichlet} more directly by proving convergence of $u_\lambda$ in $W^{2,2}(V)$.

This paper is organized as follows. In Sect. 2, we establish several necessary tools for calculus of variation on graphs, including integration by parts and properties of Sobolev spaces, especially the embedding theorems. In Sect. 3, based on the above works, we prove the existence of a ground state solution of \eqref{equation} by using the Nehari method. Finally in Sect. 4, we demonstrate the desired convergence behavior, namely as $\lambda\rightarrow +\infty$, the ground state solutions $u_{\lambda}$ of \eqref{equation} tend to $0$ outside $\Omega$ and to a ground state solution of \eqref{dirichlet} in $\Omega$. These prove Theorem \ref{existence2} and \ref{convergence}. Throughout this paper, we always assume conditions $(G_1)-(G_4)$ and $(A_1)-(A_2)$ unless otherwise stated.

\section{Preliminaries and functional settings}

In this section, we introduce some preliminaries and basic functional settings. In particular, we shall prove formulas of integration by parts and embeddings of Sobolev spaces on graphs.

Since the discreteness of graphs, the functional spaces on graphs are different from those on the Euclidean space. We first present several properties of the Sobolev spaces used in this paper.

\begin{proposition}\label{w12}
$W^{1,2}(V)$ is the completion of $C_{c}(V)$ under the norm
$\|u\|^2_{W^{1,2}(V)}=\int_{V}(|\nabla u|^{2}+u^{2})d\mu$.
\end{proposition}

\begin{proof}
We only need to prove that for any $u\in W^{1,2}(V)$, there exist $u_{k}\in C_{c}(V)$ such that $\|u_{k}-u\|_{W^{1,2}(V)}\rightarrow 0$ as $k\rightarrow\infty$.

Fix a base point $x_0\in V$ and define $\eta_{k}: V\rightarrow\mathbb{R}$ as
$$\eta_{k}(x)=\left\{\aligned &1, & d_x\leq k,\\
&\frac{2k-d_x}{k}, &k<d_x<2k,\\
&0,&d_x\geq 2k,\endaligned\right.$$
where $d_x$ denotes the distance between $x$ and $x_0$. Obviously, $\{\eta_{k}\}$ is a nondecreasing sequence of finitely supported functions which satisfies $0\leq \eta_{k}\leq 1$ and $\underset{k\rightarrow\infty}\lim\eta_{k}=1$.

Let $u_{k}=u\eta_{k}\in C_{c}(V)$. It suffices to show that
$$\|u_{k}-u\|^{2}_{W^{1,2}(V)}=\int_{V}|\nabla (u_{k}-u)|^{2}+|u_{k}-u|^{2}d\mu\rightarrow 0\ \ \hbox{as}\ \ k\rightarrow\infty.$$

Since $\int_{V}|u|^{2}d\mu<+\infty$ and $|\frac{k-d_x}{k}|<1$ for any $x\in\{x\in V, k<d_x<2k\}$, we have
\begin{eqnarray*}
\int_{V}|u_{k}-u|^{2}d\mu
&=&\underset{x\in V}\sum|u_{k}-u|^{2}(x)\mu(x)\\
&=&\underset{x\in V, d_x\leq k}\sum |u_{k}-u|^{2}(x)\mu(x)+\underset{x\in V, k<d_x<2k}\sum |u_{k}-u|^{2}(x)\mu(x)\\
&& +\underset{x\in V, d_x\geq 2k}\sum |u_{k}-u|^{2}(x)\mu(x)\\
&=&\underset{x\in V, k<d_x<2k}\sum u(x)^{2}(\frac{k-d_x}{k})^{2}\mu(x)+\underset{x\in V,d_x\geq2k}\sum u(x)^{2}\mu(x)\\
&\leq&\underset{x\in V,k<d_x<2k}\sum u(x)^{2}\mu(x)+\underset{x\in V,d_x\geq2k}\sum u(x)^{2}\mu(x)\\
&=&\underset{x\in V,d_x>k}\sum u(x)^{2}\mu(x)\\
&\rightarrow& 0 \ \ \hbox{as} \ \ k\rightarrow\infty.
\end{eqnarray*}

Next we need to prove that $\int_{V}|\nabla (u_{k}-u)|^{2}d\mu\rightarrow 0$. We have

\begin{eqnarray*}
\int_{V}|\nabla (u_{k}-u)|^{2}d\mu
&=&\underset{x\in V}\sum|\nabla (u_{k}-u)|^{2}(x)\mu(x)\\
&=&\underset{x\in V, d_x\leq k}\sum |\nabla (u_{k}-u)|^{2}(x)\mu(x)+\underset{x\in V, k< d_x< 2k}\sum |\nabla (u_{k}-u)|^{2}(x)\mu(x)\\
&&+\underset{x\in V, d_x\geq 2k}\sum |\nabla (u_{k}-u)|^{2}(x)\mu(x)\\
&=&I_{k}+II_{k}+III_{k}.
\end{eqnarray*}

By the construction of $\eta_k$ and the definition of the gradient operator, we know that
\begin{eqnarray*}
I_k&=&\sum_{x\in V, d_x\leq k}\frac{1}{2\mu(x)}\sum_{y\sim x}\omega_{xy}((u_k-u)(y)-(u_k-u)(x))^2\\
&=&\sum_{d_x=k}\frac{1}{2\mu(x)}\sum_{y\sim x; d_y=k+1}\omega_{xy}((u_k-u)(y))^2\\
&=&\frac{1}{k^2}\sum_{d_x=k}\frac{1}{2\mu(x)}\sum_{y\sim x; d_y=k+1}\omega_{xy}u^2(y)\\
&\leq& \frac{1}{k^2\mu^2_{\min}}\sum_{ d_y=k+1}u^2(y)\mu(y)\sum_{x\sim y; d_x=k}\omega_{xy}\\
&\leq &\frac{C}{k^2\mu^2_{\min}}\sum_{ d_y=k+1}u^2(y)\mu(y),
\end{eqnarray*} which tends to zero as $k\rightarrow 0$ since $u\in W^{1,2}(V)$.

\begin{eqnarray*}
III_k&=&\sum_{x\in V, d_x\geq 2k}\frac{1}{2\mu(x)}\sum_{y\sim x}\omega_{xy}((u_k-u)(y)-(u_k-u)(x))^2\\
&=&\sum_{d_x=2k}\frac{1}{2\mu(x)}\sum_{y\sim x; d_y=2k-1}\omega_{xy}((u_k-u)(y)+u(x))^2\\
&&+\sum_{d_x>2k}\frac{1}{2\mu(x)}\sum_{y\sim x }\omega_{xy}(u(y)-u(x))^2\\
&\leq&\frac{1}{k^2} \sum_{d_x=2k}\frac{1}{2\mu(x)}\sum_{y\sim x; d_y=2k-1}\omega_{xy}u^2(y)\\
&&+\sum_{d_x> 2k}\frac{1}{2\mu(x)}\sum_{y\sim x }\omega_{xy}(u(y)-u(x))^2\\
&\leq&  \frac{1}{k^2\mu^2_{\min}}\sum_{ d_y=2k-1}u^2(y)\mu(y)\sum_{x\sim y; d_x=2k}\omega_{xy}\\
&&+\sum_{d_x> 2k}|\nabla u|^2(x)\\
&\leq&\frac{C}{k^2\mu^2_{\min}}\sum_{ d_y=2k-1}u^2(y)\mu(y)+\frac{1}{\mu_{\min}}\sum_{d_x\geq 2k}|\nabla u|^2(x)\mu(x),
\end{eqnarray*} which also tends to zero as $k\rightarrow 0$ since $u\in W^{1,2}(V)$.

Now we begin to deal with $II_{k}$. Let $B=\{x\in V:k< d_x<2k\}$. We have

\begin{eqnarray}\label{ik}
II_{k}
&=&\frac{1}{2}\underset{x\in B }\sum\underset{y\sim x}\sum w_{xy}\left(u(y)\eta_k(y)-u(y)-u(x)\eta_k(x)+u(x)\right)^{2}\nonumber\\
&=&\frac{1}{2}\underset{x\in B }\sum\underset{y\sim x}\sum w_{xy}\left(u(y)-u(y)\frac{d_y}{k}+u(x)\frac{d_y}{k}-u(x)\frac{d_y}{k}
-u(x)+u(x)\frac{d_x}{k}\right)^{2}\nonumber\\
&=&\frac{1}{2}\underset{x\in B }\sum\underset{y\sim x}\sum w_{xy}\left((u(y)-u(x))(1-\frac{d_y}{k})
+u(x)(\frac{d_x}{k}-\frac{d_y}{k})\right)^{2}\nonumber\\
&\leq& \underset{x\in B }\sum\underset{y\sim x}\sum w_{xy}\left((u(y)-u(x))^{2}(1-\frac{d_y}{k})^{2}
+u(x)^{2}(\frac{d_x}{k}-\frac{d_y}{k})^{2}\right)\nonumber\\
&=&\underset{x\in B }\sum\underset{y\sim x}\sum w_{xy}(u(y)-u(x))^{2}(1-\frac{d_y}{k})^{2}+\underset{x\in B }\sum\underset{y\sim x}\sum w_{xy}u(x)^{2}(\frac{d_x}{k}-\frac{d_y}{k})^{2}\nonumber\\
&=&\underset{x\in B }\sum\underset{y\sim x}\sum w_{xy}(u(y)-u(x))^{2}(1-\frac{d_y}{k})^{2}+\underset{x\in B }\sum\underset{y\sim x}\sum w_{xy}u(x)^{2}\frac{1}{k^{2}}\nonumber\\
&=&2\underset{x\in B }\sum\frac{1}{2\mu(x)}\underset{y\sim x}\sum w_{xy}(u(y)-u(x))^{2}(1-\frac{d_y}{k})^{2}\mu(x)+\frac{1}{k^{2}}\underset{x\in B }\sum\underset{y\sim x}\sum w_{xy}u(x)^{2}\mu(x)\frac{1}{\mu(x)}\nonumber\\
&\leq& 2 C\underset{x\in B }\sum |\nabla u|^{2}(x)\mu(x)+\frac{1}{k^{2}}\frac{1}{\mu_{\min}}\underset{x\in B }\sum u(x)^{2}\mu(x)\cdot\underset{y\sim x}\sum w_{xy}.
\end{eqnarray}
The first inequality in \eqref{ik} is because of the fact that for any $a,b\in\mathbb{R}$, $(a+b)^{2}\leq 2(a^{2}+b^{2})$. In the last inequality of \eqref{ik}, we use $(1-\frac{d_y}{k})^{2}\leq C$ for some constant independent of $k$. Since $\int_{B}|\nabla u|^{2}d\mu \rightarrow 0$ and $\int_{B}u^{2}d\mu\rightarrow 0$ as $k\rightarrow\infty$, we get from the inequality \eqref{ik} and $(G_4)$ that

$$\lim_{k\rightarrow \infty}I_{k}=0.$$

Then the lemma is proved.
\end{proof}

\begin{remark}
In \cite{HuaLin,HuangKeller}, there are results similar to Proposition \ref{w12} under different assumptions on graphs.
\end{remark}

For the function space with second order derivative, we also have

\begin{proposition}\label{w22}
$W^{2,2}(V)$ is the completion of $C_{c}(V)$ under the norm
$\|u\|^2_{W^{2,2}(V)}=\int_{V}(|\Delta u|^2+|\nabla u|^{2}+u^{2})d\mu$.
\end{proposition}

\begin{proof}
Following the symbols used in the proof of Proposition \ref{w12}, we only need to prove that
\begin{equation*}\label{2order}
\lim_{k\rightarrow 0}\int_{V}|\Delta(u_k-u)|^2 d\mu= 0.
\end{equation*}

By directly computations, we have

\begin{eqnarray*}
\int_{V}|\Delta(u_k-u)|^2 d\mu
&=&\underset{x\in V, d_x\leq k }\sum |\Delta(u\eta_k-u)|^2(x)\mu(x)
+\underset{x\in V, k< d_x< 2k }\sum |\Delta(u\eta_k-u)|^2(x)\mu(x)\\
&&+\underset{x\in V, d_x\geq 2k }\sum |\Delta(u\eta_k-u)|^2(x)\mu(x)\\
&=&I_k+II_k+III_k.
\end{eqnarray*}

For the second term $II_k$, we have

\begin{eqnarray}\label{ik2}
II_{k}
&=&\underset{x\in B }\sum \mu(x)\left(\frac{1}{\mu(x)}\underset{y\sim x}\sum w_{xy}(u(y)\eta_k(y)-u(y)-u(x)\eta_k(x)+u(x))\right)^{2}\nonumber\\
&=&\underset{x\in B }\sum \frac{1}{\mu(x)}\left(\underset{y\sim x}\sum w_{xy}\left(u(y)-u(y)\frac{d_y}{k}+u(x)\frac{d_y}{k}-u(x)\frac{d_y}{k}
-u(x)+u(x)\frac{d_x}{k}\right)\right)^{2}\nonumber\\
&=&\underset{x\in B }\sum \frac{1}{\mu(x)}\left(\underset{y\sim x}\sum w_{xy}((u(y)-u(x))(1-\frac{d_y}{k})
+u(x)(\frac{d_x}{k}-\frac{d_y}{k}))\right)^{2}\nonumber\\
&\leq& 2\underset{x\in B }\sum \frac{1}{\mu(x)}\left(
(\underset{y\sim x}\sum w_{xy}(u(y)-u(x))(1-\frac{d_y}{k}))^{2}
+(\underset{y\sim x}\sum w_{xy}u(x)(\frac{d_x}{k}-\frac{d_y}{k}))^{2}\right)\nonumber\\
&\leq& 2\underset{x\in B}\sum \frac{1}{\mu(x)}\left(
\underset{y\sim x}\sum w_{xy}(u(y)-u(x))^2
\underset{y\sim x}\sum w_{xy}(1-\frac{d_y}{k})^2
+\underset{y\sim x}\sum w_{xy}u^2(x)
\underset{y\sim x}\sum w_{xy}(\frac{d_x}{k}-\frac{d_y}{k})^2\right)\nonumber\\
&=&4\underset{x\in B }\sum \frac{1}{2\mu(x)}
\underset{y\sim x}\sum w_{xy}(u(y)-u(x))^2
\underset{y\sim x}\sum w_{xy}(1-\frac{d_y}{k})^2
+2\underset{x\in B }\sum\frac{1}{\mu(x)}\underset{y\sim x}\sum w_{xy}u^2(x)
\underset{y\sim x}\sum w_{xy}\frac{1}{k^2}\nonumber\\
&=&4\underset{x\in B }\sum |\nabla u|^2(x)\underset{y\sim x}\sum w_{xy}(1-\frac{d_y}{k})^2
+\frac{2}{k^2}\underset{x\in B }\sum\frac{1}{\mu(x)}u^2(x)(\underset{y\sim x}\sum w_{xy})^2\nonumber\\
&\leq&\frac{4C}{\mu_{\min}}\underset{x\in B }\sum |\nabla u|^2(x)\mu(x)
+\frac{2C^2}{k^2\mu^2_{\min}}\underset{x\in B }\sum u^2(x)\mu(x)\nonumber\\
&=&\frac{4C}{\mu_{\min}}\int_B |\nabla u|^2d\mu
+\frac{2C^2}{\mu^2_{\min}k^2}\int_B |u|^2d\mu.
\end{eqnarray}

In the first and the second inequalities of \eqref{ik2}, we use the facts that $(a+b)^2\leq 2(a^2+b^2)$ and $(\sum_i a_i b_i)^2\leq \sum_i a_i^2\cdot \sum_i b_i^2 $, where $a,b, a_i, b_i\in \mathbb{R}$. Since $|1-\frac{d_y}{k}|<1$ and $(G_3)-(G_4)$, we get the last inequality of \eqref{ik2}. By $u\in W^{2,2}(V)$, we get $\underset{k\rightarrow \infty}\lim II_k=0$.

The computations for the first and third terms are similar to those in Proposition \ref{w12} and for brevity, we omit them here.
\end{proof}

\begin{remark}
Let $\Omega$ be a connected and bounded domain in $V$. For the space $W^{1,2}_{0}(\Omega)$, which is the completion of $C_{c}(\Omega)$ under the norm $\|u\|^2_{W^{1,2}_{0}(\Omega)}=\int_{\Omega\cup\partial\Omega}|\nabla u|^{2}d\mu+\int_{\Omega}|u|^{2}d\mu$, we have that it is just $C_c(\Omega)$, namely, $W^{1,2}_0(\Omega)=C_c(\Omega)$. Furthermore, for any integer $m>2$, we also have $W_0^{m,2}(\Omega)=C_c^{m-1}(\Omega)$, where $C^{m-1}_{c}(\Omega):=\{u:\Omega\rightarrow\mathbb{R}:u(x)=|\nabla u|=\cdots=|\nabla^{m-1}u|=0 \ \ \hbox{on}\ \ \partial\Omega\}$. Since a bounded domain $\Omega\subset V$ just contains finite vertices, the proofs are easy and we leave them to the readers.
\end{remark}

Next, we turn to formulas of integration by parts on graphs, which are fundamental when we use methods in calculus of variations. The proofs of the next two lemmas can be found in \cite{ZhangZhao} and we omit them here.

\begin{lemma}\label{a}
Suppose that $u\in W^{1,2}(V)$. Then for any $v\in C_{c}(V)$, we have
\begin{equation*}\label{partw12v}
\int_{V}\nabla u\nabla vd\mu=\int_{V}\Gamma(u,v)d\mu=-\int_{V}(\Delta u)vd\mu.
\end{equation*}
\end{lemma}

\begin{lemma}\label{b}
Suppose that $u\in W^{1,2}(V)$ and $\Omega\subset V$ is a bounded connected domain. Then for any $v\in C_{c}(\Omega)$, we have
\begin{equation*}\label{partw12omega}
\int_{\Omega\cup\partial\Omega}\nabla u\nabla vd\mu=\int_{\Omega\cup\partial\Omega}\Gamma(u,v)d\mu=-\int_{\Omega}(\Delta u)vd\mu.
\end{equation*}
\end{lemma}

Since we consider the biharmonic equations, we shall also generalize these two lemmas to the higher order case.

\begin{lemma}\label{c}
Suppose that $u\in W^{2,2}(V)$. Then for any $v\in C_{c}(V)$ we have
\begin{equation*}\label{partw22v}
\int_{V}(\Delta^{2} u)vd\mu=\int_{V}\Delta u\Delta vd\mu,
\end{equation*}
where $\Delta^{2} u=\Delta(\Delta u)$.
\end{lemma}

\begin{proof}
By using Lemma \ref{a}, we have
\begin{eqnarray*}
\int_{V}(\Delta^{2} u) v d\mu
&=&\int_{V}\Delta(\Delta u)v d\mu\\
&=&-\int_{V}\nabla(\Delta u) \nabla v d\mu\\
&=&-\frac{1}{2}\underset{x\in V}\sum\underset{y\sim x}\sum\omega_{xy}(v(y)-v(x))(\Delta u(y)-\Delta u(x))\\
&=&-\frac{1}{2}\underset{x\in V}\sum\underset{y\sim x}\sum\omega_{xy}(v(y)-v(x))\Delta u(y)+\frac{1}{2}\underset{x\in V}\sum\underset{y\sim x}\sum\omega_{xy}(v(y)-v(x))\Delta u(x)\\
&=&-\frac{1}{2}\underset{y\in V}\sum\underset{x\sim y}\sum\omega_{xy}(v(y)-v(x))\Delta u(y)+\frac{1}{2}\underset{x\in V}\sum\underset{y\sim x}\sum\omega_{xy}(v(y)-v(x))\Delta u(x)\\
&=&-\frac{1}{2}\underset{y\in V}\sum\underset{x\sim y}\sum\omega_{xy}(v(y)-v(x))\Delta u(y)+\frac{1}{2}\int_{V}\Delta u\Delta vd\mu\\
&=&\frac{1}{2}\underset{x\in V}\sum\underset{y\sim x}\sum\omega_{xy}(v(y)-v(x))\Delta u(x)+\frac{1}{2}\int_{V}\Delta u\Delta vd\mu\\
&=&\int_{V}\Delta u\Delta vd\mu.
\end{eqnarray*}
\end{proof}

\begin{lemma}\label{d}
Suppose that $u\in W^{2,2}(V)$. Then for any $v\in C_{c}(\Omega)$, we have
\begin{equation*}\label{partw22omega}
\int_{\Omega}(\Delta^{2} u)v d\mu=\int_{\Omega\cup\partial\Omega}\Delta u\Delta v d\mu.
\end{equation*}
\end{lemma}

\begin{proof}
By using Lemma \ref{c}, we only need prove that

\begin{eqnarray*}
\int_{V\setminus\Omega\cup\partial\Omega}\Delta u\Delta v=0.
\end{eqnarray*}

Since $v\in C_{c}(\Omega)$, $v=0$ on $V\setminus\Omega$. Thus, for any $x\in V\setminus\{\Omega\cup\partial\Omega\}$, there hold $v(x)=0$ and $v(y)=0$ for all $y\sim x$. Therefore, we get for $x\in V\setminus\{\Omega\cup\partial\Omega\}$
\begin{eqnarray*}
\Delta v(x)=\frac{1}{\mu(x)}\sum_{y\sim x}\omega_{xy}(v(y)-v(x))=0,
\end{eqnarray*}
and the lemma is proved.
\end{proof}

Now we can define the weak solution of the equation \eqref{equation} as

\begin{definition}\label{defofws}
Suppose $u\in E_{\lambda}$. If for any $\phi\in E_\lambda$, there holds
\begin{equation*}\label{weaksolution}
\int_{V}(\Delta u\Delta\phi+\nabla u\nabla\phi+(\lambda a+1)u\phi)d\mu=\int_{V}|u|^{p-2}u\phi d\mu,
\end{equation*}
then $u$ is called a weak solution of \eqref{equation}.
\end{definition}

Similarly, the weak solution of the equation \eqref{dirichlet} is defined as

\begin{definition}
Suppose $u\in H(\Omega)$. If for any $\phi\in H(\Omega)$, there holds
\begin{equation*}\label{weaksolutiondirichlet}
\int_{\Omega\cap\partial\Omega}\Delta u\Delta\phi +\nabla u\nabla\phi d\mu+\int_{\Omega}u\phi d\mu=\int_{\Omega}|u|^{p-1}u\phi d\mu,
\end{equation*}
then $u$ is called a weak solution of \eqref{dirichlet}.
\end{definition}

Finally in this section, we come to the Sobolev embedding theorems on the graphs. Since we are concerned with the behavior of solutions $u_{\lambda}$ of \eqref{equation} as $\lambda\rightarrow +\infty$, without loss of generality, we can assume that $\lambda>1$. For brevity, We use $\|\cdot\|_{q,V}$ and $\|\cdot\|_{q,\Omega}$ to denote the $L^{q}$ norms on $V$ and $\Omega$ respectively and we sometimes omit the subscripts $V$ and $\Omega$ if it is clear from the context.
\begin{lemma}\label{e}
Assume that $a(x)$ satisfies $(A_{1})$ and $(A_{2})$. Then $E_{\lambda}$ is continuously embedded into $L^{q}(V)$ for any $q\in [2,+\infty]$ and the embedding is independent of $\lambda$. Namely, there exists a constant $\eta_{q}$ depending only on $q$ such that for any $u\in E_{\lambda}$,
\begin{equation*}
\|u\|_{q,V}\leq \eta_{q}\|u\|_{E_{\lambda}}.
\end{equation*}
Moreover, for any bounded sequence $\{u_{k}\}\subset E_{\lambda}$, there exists $u\in E_{\lambda}$ such that, up to a subsequence,
$$\left\{\aligned&u_{k} \rightharpoonup u  &\hbox{in} \ \ E_{\lambda};\\
&u_{k}(x)\rightarrow u(x) &\forall x\in V;\\
&u_{k} \rightarrow u &\hbox{in}\ \ L^{q}(V).
\endaligned\right.$$
\end{lemma}

\begin{proof}
For any $u\in E_{\lambda}$ and vertex $x_{0}\in V$, we have
$$\aligned \|u\|^{2}_{E_{\lambda}}&=\int_{V}(|\Delta u|^{2}+|\nabla u|^{2}+(\lambda a+1)u^{2})d\mu\\
&\geq\int_{V}u^{2}d\mu\\
&=\underset{x\in V}\sum u(x)^{2}\mu(x)\\
&\geq\mu_{\min}u(x_{0})^{2},\endaligned$$
which gives
\begin{equation*}
u(x_{0})\leq\left(\frac{1}{\mu_{\min}}\right)^{\frac{1}{2}}\|u\|_{E_{\lambda}}.
\end{equation*}
Therefore, $E_{\lambda} \hookrightarrow L^{\infty}(V)$ continuously and the embedding is independent of $\lambda.$ Thus $E_{\lambda} \hookrightarrow L^{q}(V)$ continuously for any $2\leq q<\infty.$ In fact, for any $u\in E_{\lambda}$, we have $u\in L^{2}(V)$. Then, for any $2\leq q<\infty$,
$$\int_{V}|u|^{q}d\mu=\int_{V}|u|^{2}|u|^{q-2}d\mu\leq(\mu_{\min})^{\frac{2-q}{2}}\|u\|^{q-2}_{E_{\lambda}}\int_{V}|u|^{2}d\mu<+\infty,$$
which implies that $u\in L^{q}(V)$, $2\leq q<\infty.$

Since $E_{\lambda}$ is a Hilbert space, it is reflexive. Thus for any bounded sequence $\{u_{k}\}$ in $E_{\lambda}$, we have that, up to a subsequence, $u_{k} \rightharpoonup u$ in $E_{\lambda}$. On the other hand, $\{u_{k}\}\subset E_{\lambda}$ is also bounded in $L^{2}(V)$ and we get $u_{k} \rightharpoonup u$ in $L^{2}(V)$, which tells us that, for any $v\in L^{2}(V)$,
\begin{equation}\label{pointwise}
\underset{k\rightarrow\infty}\lim\int_{V}(u_{k}-u)vd\mu=\underset{k\rightarrow\infty}\lim\underset{x\in V}\sum\mu(x)(u_{k}(x)-u(x))v(x)=0.
\end{equation}
Take any $x_{0}\in V$ and let
$$v_{0}(x)=\left\{\aligned&1 &x=x_{0},\\
&0 &x\neq x_{0}.\endaligned\right.$$
Obviously, $v_0$ belongs to $L^{2}(V)$.
By substituting $v_{0}$ into \eqref{pointwise}, we get
$$\underset{k\rightarrow\infty}\lim\mu(x_{0})(u_{k}(x_{0})-u(x_{0}))=0,$$
which implies that $\underset{k\rightarrow\infty}\lim u_{k}(x)=u(x)$ for any $x\in V.$

Next we prove $u_{k}\rightarrow u$ in $L^{q}(V)$ for all $2\leq q\leq +\infty.$ Since $\{u_{k}\}$ bounded in $E_{\lambda}$ and $u\in E_{\lambda}$, there exists some constant $C_{1}$ such that
$$\|u_{k}-u\|^{2}_{E_{\lambda}}\leq C_{1}.$$
Let $x_{0}\in V$ be fixed. For any $\epsilon>0$, in view of $(A_{2})$, there exists some $R>0$ such that when $d(x,x_{0})>R$
$$a(x)\geq\frac{C_{1}}{\epsilon}.$$
Noticing $\lambda>1$, we have
\begin{eqnarray}\label{greatthanr}
\int_{d(x,x_{0})>R}|u_{k}-u|^{2}d\mu
&\leq&\frac{\epsilon}{C_{1}}\int_{d(x,x_{0})>R}a|u_{k}-u|^{2}d\mu\nonumber\\
&\leq&\frac{\epsilon}{C_{1}} \|u_{k}-u\|^{2}_{E_{\lambda}}\nonumber\\
&\leq&\epsilon.
\end{eqnarray}
Moreover, up to a subsequence, we have
\begin{equation}\label{leqr}
\underset{k\rightarrow +\infty}\lim\int_{d(x,x_{0})\leq R}|u_{k}-u|^{2}d\mu=0.
\end{equation}
Combining \eqref{greatthanr} and \eqref{leqr}, we conclude
$$\underset{k\rightarrow +\infty}\liminf\int_{V}|u_{k}-u|^{2}d\mu=0.$$
In particular, up to a subsequence, there holds $u_{k}\rightarrow u$ in $L^{2}(V)$. Since
$$\|u_{k}-u\|^{2}_{L^{\infty}(V)}\leq\frac{1}{\mu_{\min}}\int_{V}|u_{k}-u|^{2}d\mu,$$
for any $2<q<+\infty$, we have
$$\int_{V}|u_{k}-u|^{q}d\mu\leq\|u_{k}-u\|^{q-2}_{L^{\infty}(V)}\int_{V}|u_{k}-u|^{2}d\mu\rightarrow 0\ \ \hbox{as} \ \ k\rightarrow\infty.$$
Therefore, up to a subsequence, $u_{k}\rightarrow u$ in $L^{q}(V)$ for all $2\leq q\leq +\infty.$
\end{proof}

\begin{lemma}\label{f}
Assume that $\Omega$ is a bounded domain in $V$. Then $H(\Omega)$ is compactly embedded into $L^{q}(\Omega)$ for any $q\in [1,+\infty]$. In particular, there exists a constant $C$ depending only on $q$ such that for any $u\in H(\Omega)$,
$$\|u\|_{q,\Omega}\leq C\|u\|_{H(\Omega)}.$$
Moreover, $H(\Omega)$ is  pre-compact. Namely, if $u_{k}$ is bounded in $H(\Omega)$,
up to a subsequence, there exists some $u\in H(\Omega)$ such that $u_{k}\rightarrow u$ in $H(\Omega).$
\end{lemma}

\begin{proof} Since $\Omega$ is a finite set in $V$, $H(\Omega)$ is a finite dimensional space. Therefore the conclusions of the lemma are obvious and we omit the proofs here. \end{proof}

\section{The existence of ground state solutions }

In this section we prove the existence of ground state solutions of \eqref{equation} by the Nehari method. First we prove the following lemmas.

\begin{lemma}\label{g}
If $u\in E_{\lambda}$ is a weak solution of \eqref{equation}, $u$ is also a point-wise solution of the equation.
\end{lemma}

\begin{proof}
Since $u\in E_{\lambda}$ is a weak solution of \eqref{equation},
for any $\phi\in E_{\lambda}$, there holds
$$
\int_{V}(\Delta u\Delta \phi+\nabla u\nabla\phi+(\lambda a+1)u\phi)d\mu=\int_{V}|u|^{p-2}u\phi d\mu.
$$
Then, by Lemma \ref{c} we have
\begin{equation}\label{test1}
\int_{V} \Delta^{2} u\phi+\nabla u\nabla\phi+(\lambda a+1)u\phi)d\mu=\int_{V}|u|^{p-2}u\phi d\mu,\ \ \forall \phi\in C_{c}(V).
\end{equation}
For any fixed $x_{0}\in V$, taking a test function $\phi: V\rightarrow \mathbb{R}$ in \eqref{test1} with
 $$\phi(x)= \left\{\aligned &1, &x=x_{0},\\
&0, &x\neq x_{0},\endaligned\right.$$
we have
$$
\Delta^{2} u(x_{0})-\Delta u(x_{0})+(\lambda a(x_{0})+1)u(x_{0})-|u(x_{0})|^{p-2}u(x_{0})=0.
$$
Since $x_{0}$ is arbitrary, we conclude that $u$ is a point-wise solution of \eqref{equation}.
\end{proof}

\begin{remark}\label{boundedpointwise}
Similarly, if $u\in H(\Omega)$ is a weak solution of \eqref{dirichlet}, then $u$ is also a point-wise solution of the equation \eqref{dirichlet}.
\end{remark}

\begin{lemma}\label{h}
$\mathcal{N}_{\lambda}$ is non-empty.
\end{lemma}

\begin{proof}
For a fixed $u\in E_{\lambda}\setminus\{0\}$, we define a function $g(t)$ on $\mathbb{R}$ as
$$
g(t)=J^{'}_{\lambda}(tu)tu=t^{2}\int_{V}(|\Delta u|^{2}+|\nabla u|^{2}+(\lambda a+1)u^{2})d\mu-t^{p}\int_{V}|u|^{p}d\mu.
$$
Since $p> 2$ and $u\not\equiv0$, there exists $t_{0}\in (0,+\infty)$ such that $g(t_{0})=0$ which implies that $t_{0}u\in\mathcal{N}_{\lambda}.$
\end{proof}

\begin{lemma}\label{i}
$m_{\lambda}=\underset{u\in \mathcal{N}_{\lambda}}\inf J_{\lambda}(u)>0.$
\end{lemma}

\begin{proof}
Since $u\in \mathcal{N}_{\lambda}$, we have
$$\|u\|^{2}_{E_{\lambda}}=\int_{V}|u|^{p}d\mu.$$
By Lemma \ref{e}, we have
 $$\|u\|^{2}_{E_{\lambda}}=\|u\|^{p}_{p}\leq\eta^{p}_{p}\|u\|^{p}_{E_{\lambda}}.$$
Since $p>2$, we get
$$
\|u\|_{E_{\lambda}}\geq\left(\frac{1}{\eta_{p}}\right)^{\frac{p}{p-2}}>0,
$$
which gives that
$$
m_{\lambda}=\inf_{u\in \mathcal{N}_{\lambda}}J_{\lambda}(u)=(\frac{1}{2}-\frac{1}{p})\inf_{u\in \mathcal{N}_{\lambda}}\|u\|^{2}_{E_{\lambda}}\geq\left(\frac{1}{2}-\frac{1}{p}\right)\left(\frac{1}{\eta_{p}}\right)^{\frac{2p}{p-2}}>0.
$$
\end{proof}

\begin{lemma}\label{j}
$m_{\lambda}$ can be achieved by some $u_{\lambda}\in \mathcal{N}_{\lambda}$.
\end{lemma}

\begin{proof}
Take a sequence $\{u_{k}\}\subset \mathcal{N}_{\lambda}$ such that $\underset{k\rightarrow +\infty}\lim J_{\lambda}(u_{k})=m_{\lambda}.$ We claim that $\{u_{k}\}$ is bounded in $E_{\lambda}$.
In fact, $\{u_{k}\}\subset \mathcal{N}_{\lambda}$ and $\underset{k\rightarrow +\infty}\lim J_{\lambda}(u_{k})=m_{\lambda}$ are equivalent to
$$
J^{'}_{\lambda}(u_{k})u_{k}=\|u_{k}\|^{2}_{E_{\lambda}}-\|u_{k}\|^{p}_{p}=0,
$$
$$
\frac{1}{2}\|u_{k}\|^{2}_{E_{\lambda}}-\frac{1}{p}\|u_{k}\|^{p}_{p}=m_{\lambda}+o_{k}(1).
$$
Then we have
\begin{equation}\label{ukbound}
m_{\lambda}+o_{k}(1)=(\frac{1}{2}-\frac{1}{p})\|u_{k}\|^{2}_{E_{\lambda}},
\end{equation}
Since $p>2$, \eqref{ukbound} implies that $\{u_{k}\}$ is bounded in $E_{\lambda}$.

Lemma \ref{e} tells us that there exists some $u\in E_{\lambda}$ such that, up to a subsequence,
$$\left\{\aligned&u_{k} \rightharpoonup u_{\lambda}  &\hbox{in} \ \ E_{\lambda};\\
&u_{k}(x)\rightarrow u_{\lambda}(x) &\forall x\in V;\\
&u_{k} \rightarrow u_{\lambda} &\hbox{in}\ \ L^{p}(V),
\endaligned\right.$$
as $k\rightarrow +\infty$. By weak lower semi-continuity of the norm for $E_{\lambda}$ and convergence of $u_{k}$ to $u_{\lambda}$ in $L^{p}(V )$, we have
\begin{eqnarray}\label{ulambdaconvergence} J_{\lambda}(u_{\lambda})
&=&\frac{1}{2}\|u_{\lambda}\|^{2}_{E_{\lambda}}-\frac{1}{p}\|u_{\lambda}\|^{p}_{p}
\nonumber\\
&\leq&\underset{k\rightarrow +\infty}\liminf(\frac{1}{2}\|u_{k}\|^{2}_{E_{\lambda}}-\frac{1}{p}\|u_{k}\|^{p}_{p})
\nonumber\\
&=&\underset{k\rightarrow +\infty}\liminf J_{\lambda}(u_{k})\nonumber\\
&=&m_{\lambda}.
\end{eqnarray}
Now, we only need to show that $u_{\lambda}\in\mathcal{N}_{\lambda}$. Up to a subsequence, we can assume that $\|u_{k}\|^{2}_{E_{\lambda}}\rightarrow C>0$ for some positive constant $C$ as $k\rightarrow \infty.$ This together with $\|u_{k}\|^{2}_{E_{\lambda}}=\|u_{k}\|^{p}_{p}$ gives that
$$
\|u_{\lambda}\|^{p}_{p}=\underset{k\rightarrow \infty}\lim\|u_{k}\|^{p}_{p}=C.
$$
Noticing that $u_{k}\in\mathcal{N}_{\lambda}$, we have
$$
\|u_{\lambda}\|^{2}_{E_{\lambda}}
\leq\underset{k\rightarrow\infty}\liminf\|u_{k}\|^{2}_{E_{\lambda}}
=\underset{k\rightarrow\infty}\liminf\|u_{k}\|^{p}_{p}=\|u_{\lambda}\|^{p}_{p}.
$$
Suppose that $\|u_{\lambda}\|^{2}_{E_{\lambda}}<\|u_{\lambda}\|^{p}_{p}$. By similar arguments as in Lemma \ref{h}, we know that there exists some $t\in (0,1)$ such that $tu_{\lambda}\in\mathcal{N}_{\lambda}$. This gives that
\begin{eqnarray*}
0<m_{\lambda}\leq J_{\lambda}(tu_{\lambda})
&=&(\frac{1}{2}-\frac{1}{p})\|tu_{\lambda}\|^{2}_{E_{\lambda}}\\
&=&t^{2}(\frac{1}{2}-\frac{1}{p})\|u_{\lambda}\|^{2}_{E_{\lambda}}\\
&\leq& t^{2}\underset{k\rightarrow\infty}\liminf(\frac{1}{2}-\frac{1}{p})\|u_{k}\|^{2}_{E_{\lambda}}\\
&=&t^{2}\underset{k\rightarrow\infty}\liminf J_{\lambda}(u_{k})\\
&=&t^{2}m_{\lambda}<m_{\lambda},
\end{eqnarray*}
which is a contradiction with the fact that $m_{\lambda}=\underset{u\in \mathcal{N}_{\lambda}}\inf J_{\lambda}(u)$. Therefore we have $\|u_{\lambda}\|^{2}_{E_{\lambda}}=\|u_{\lambda}\|^{p}_{p}$ and $u_{\lambda}\in\mathcal{N}_{\lambda}$. This together with \eqref{ulambdaconvergence} gives that $m_{\lambda}$ is achieved by $u_{\lambda}.$
\end{proof}

The next lemma completes the proof of Theorem \ref{existence1}.

\begin{lemma}\label{k}
$u_{\lambda}\in\mathcal{N}_{\lambda}$ is a ground state solution of \eqref{equation}.
\end{lemma}

\begin{proof}
We shall prove that for any $\phi\in C_{c}(V)$, there holds
$$J_{\lambda}^{'}(u_{\lambda})\phi=0.$$
Since $u_\lambda\not\equiv 0$, we can choose a constant $\epsilon>0$ such that $u_{\lambda}+s\phi\not\equiv 0$ for any $s\in (-\epsilon,\epsilon)$. Furthermore, for any $s\in (-\epsilon,\epsilon)$, there exists some $t(s)\in (0,\infty)$, such that $t(s)(u_{\lambda}+s\phi)\in\mathcal{N}_{\lambda}$. In fact, $t(s)$ can be taken as
$$
t(s)=\left(\frac{\|u_{\lambda}+s\phi\|^{2}_{E_{\lambda}}}
{\|u_{\lambda}+s\phi\|^{p}_{p}}\right)^{\frac{1}{p-2}},
$$
and in particular, we have $t(0)=1$. Define a function $\gamma(s):(-\epsilon,\epsilon)\rightarrow\mathbb{R}$ as
$$
\gamma(s):=J_{\lambda}(t(s)(u_{\lambda}+s\phi)).
$$
Since $t(s)(u_{\lambda}+s\phi)\in\mathcal{N}_{\lambda}$ and $J_{\lambda}(u_{\lambda})=\underset{u\in \mathcal{N}_{\lambda}}\inf J_{\lambda}(u)$, $\gamma(s)$ achieves its minimum at $s=0$. This implies that
\begin{eqnarray*} 0=\gamma^{'}(0)
&=&J_{\lambda}^{'}(t(0)u_{\lambda})\cdot[t^{'}(0)u_{\lambda}+t(0)\phi]\\
&=&J_{\lambda}^{'}(u_{\lambda})\cdot t^{'}(0)u_{\lambda}+J_{\lambda}^{'}(u_{\lambda})\cdot\phi\\
&=&J_{\lambda}^{'}(u_{\lambda})\cdot\phi,
\end{eqnarray*}
where in the last equality, we have used the fact that $u_{\lambda}\in\mathcal{N}_{\lambda}$  and $J_{\lambda}^{'}(u_{\lambda})\cdot u_{\lambda}=0$.
\end{proof}

\section{Convergence of the ground state solutions}

In this section, we prove that the ground state solutions $u_{\lambda}$  of \eqref{equation} converge to a ground state solution of \eqref{dirichlet} as $\lambda\rightarrow +\infty$, which also implies Theorem \ref{existence2}.

\begin{lemma}\label{nu}
There exists a constant $\nu>0$ which is independent of $\lambda$, such that for any critical point $u\in E_{\lambda}\backslash\{0\}$ of $J_{\lambda}$, we have $\|u\|_{E_{\lambda}}\geq\nu$.
\end{lemma}

\begin{proof}
Since $u$ is a critical point of $J_{\lambda}$, we have
\begin{eqnarray*}
0=J'_{\lambda}(u)u&=&\|u\|^{2}_{E_{\lambda}}-\int_{V}|u|^{p}d\mu\\
&\geq&\|u\|^{2}_{E_{\lambda}}-\eta^{p}_{p}\|u\|^{p}_{E_{\lambda}},
\end{eqnarray*}
where we have used Lemma \ref{e} to get the inequality.
Then we can choose $\nu=\left(\frac{1}{\eta_{p}}\right)^{\frac{p}{p-2}}$ and the lemma is proved.
\end{proof}

Next, we prove a lemma for the $(PS)_{c}$ sequence of $J_{\lambda}$.

\begin{lemma}\label{pslevel}
For any $(PS)_{c}$ sequence $\{u_{k}\}$ of $J_{\lambda}$, there holds
\begin{equation}\label{level}
\underset{k\rightarrow +\infty}\lim \|u_{k}\|^{2}_{E_{\lambda}}=\frac{2p }{p-2}c.
\end{equation}
Moreover, there exists a positive constant $C_{1}>0$ which is independent of $\lambda$, such that either $c>C_1$ or $c=0$.
\end{lemma}

\begin{proof}
Since $J_{\lambda}(u_{k})\rightarrow c$ and $J'_{\lambda}(u_{k})\rightarrow 0$ as $k\rightarrow +\infty$, we have
\begin{eqnarray*}
c&=&\underset{k\rightarrow +\infty}\lim(J_{\lambda}(u_{k})-\frac{1}{p}J'_{\lambda}(u_{k})u_{k})\\
&=&\underset{k\rightarrow +\infty}\lim(\frac{1}{2}-\frac{1}{p})\|u_{k}\|^{2}_{E_{\lambda}}\\
&=&\frac{p-2}{2p }\underset{k\rightarrow +\infty}\lim\|u_{k}\|^{2}_{E_{\lambda}},
\end{eqnarray*}
which gives \eqref{level}.

By Lemma \ref{e}, for any $u\in E_{\lambda}$, we have
\begin{equation}\label{level1}
J^{'}_{\lambda}(u)u=\|u\|^{2}_{E_{\lambda}}-\|u\|^{p}_{p}\geq \|u\|^{2}_{E_{\lambda}}-\eta^{p}_{p}\|u\|^{p}_{E_{\lambda}}.
\end{equation}
Take $\rho=(\frac{1}{2\eta^{p}_{p}})^{\frac{1}{p-2}}$. If $\|u\|_{E_{\lambda}}\leq\rho$, we get from \eqref{level1}
\begin{equation}\label{level2}
J'_{\lambda}(u)u\geq \frac{1}{2}\|u\|^{2}_{E_{\lambda}}.
\end{equation}
Take $C_{1}=\frac{p-2}{2p}\rho^{2}$ and suppose $c<C_{1}$. Since $\{u_{k}\}$ is a $(PS)_{c}$ sequence, \eqref{level} gives
$$\underset{k\rightarrow +\infty}\lim \|u_ {k}\|^{2}_{E_{\lambda}}=\frac{2p }{p-2}c<\frac{2p }{p-2}C_{1}=\rho^{2}.$$
From \eqref{level2}, we have that for large $k$, there holds
$$\frac{1}{2}\|u_{k}\|^{2}_{E_{\lambda}}\leq J'_{\lambda}(u_{k})u_{k}=o_{k}(1)\|u_{k}\|_{E_{\lambda}},$$
which implies that $\|u_{k}\|_{E_{\lambda}}\rightarrow 0$ as $k\rightarrow +\infty$. It follows immediately that $J_{\lambda}(u_{k})\rightarrow c=0$ and the desired results are proved for $C_{1}=\frac{p-2}{2p}\rho^{2}=\frac{p-2}{2p}(\frac{1}{2\eta^{p+1}_{p+1}})^{\frac{2}{p-2}}.$
\end{proof}

\begin{remark}\label{psbound}
By the proof of the existence of a ground state solutions $u_{\lambda}$, we know that there exists a $(PS)_{c}$ sequence $\{u_{k}\}$ converges weakly to $u_{\lambda}$ in $E_{\lambda}$ with $c=m_{\lambda}$. By weak lower semi-continuity of the norm $\|\cdot\|_{E_{\lambda}}$, we get that $\|u_{\lambda}\|_{E_{\lambda}}$ is bounded by $\frac{2p m_{\lambda}}{p-2}.$
\end{remark}

For the ground states $m_{\lambda}$ and $m_{\Omega}$, we have

\begin{lemma}\label{groundstate}
$m_{\lambda}\rightarrow m_{\Omega}$ as $\lambda\rightarrow\infty.$
\end{lemma}

\begin{proof}
Since $\mathcal{N}_{\Omega}\subset\mathcal{N}_{\lambda}$, we obviously have that $m_{\lambda}\leq m_{\Omega}$ for any $\lambda>0$. Take a sequence $\lambda_{k}\rightarrow\infty$ such that
$$
\underset{k\rightarrow\infty}\lim m_{\lambda_{k}}=M\leq m_{\Omega},
$$
where $m_{\lambda_{k}}$ is the ground state and $u_{\lambda_{k}}\in\mathcal{N}_{\lambda_{k}}$ is the corresponding ground state solution. Lemma \ref{pslevel} tells us that $m_\Omega\geq M>0$.

By Remark \ref{psbound}, $\{u_{\lambda_{k}}\}$ is uniformly bounded in $W^{2,2}(V)$. Up to a subsequence, we can assume that there exists some $u_{0}\in W^{2,2}(V)$ such that
$$u_{\lambda_{k}}\rightharpoonup u_{0}\ \ \hbox{in}\ \ W^{2,2}(V)$$
and for any $q\in [2,+\infty)$,
\begin{equation}\label{lpconvergence}
u_{\lambda_{k}}\rightarrow u_{0}\ \ \hbox{in}\ \ L^{q}(V).
\end{equation}
We claim that $u_{0}|_{\Omega^{c}}=0.$ Otherwise, there exists a vertex $x_{0}\notin\Omega$ such that $u_{0}(x_{0})\neq 0$. Since $u_{\lambda_{k}}\in\mathcal{N}_{\lambda_{k}}$, we have
$$
J_{\lambda}(u_{\lambda_{k}})
=(\frac{1}{2}-\frac{1}{p})\|u_{\lambda_{k}}\|^{2}_{E_{\lambda_{k}}}
\geq\frac{p-2}{2p}\lambda_{k}\int_{V}a|u_{\lambda_{k}}|^{2}d\mu\\
\geq\frac{p-2}{2p}\lambda_{k}a(x_{0})|u_{\lambda_{k}}(x_{0})|^{2}\mu(x_{0}).
$$
Since $a(x_{0})>0$, $\mu(x_{0})\geq\mu_{\min}>0,$ $u_{\lambda_{k}}(x_{0})\rightarrow u_{0}(x_{0})\neq 0$ and $\lambda_{k}\rightarrow\infty$, we get
$$\underset{k\rightarrow\infty}\lim J_{\lambda_{k}}(u_{\lambda_{k}})=\infty,$$
which is a contradiction to the fact that $m_{\lambda_{k}}<m_{\Omega}.$

By weak lower semi-continuity of the norm $\|\cdot\|_{W^{2,2}(V)}$ and \eqref{lpconvergence}, we get
\begin{eqnarray*}
&&\int_{\Omega\cup\partial\Omega}(|\Delta u_{0}|^{2}+|\nabla u_{0}|^{2})d\mu+\int_{\Omega}|u_{0}|^{2}d\mu\\
&\leq&\int_{V}(|\Delta u_{0}|^{2}+|\nabla u_{0}|^{2}+|u_{0}|^{2})d\mu\\
&\leq&\underset{k\rightarrow+\infty}\liminf \int_{V}(|\Delta u_{\lambda_{k}}|^{2}+|\nabla u_{\lambda_{k}}|^{2}+|u_{\lambda_{k}}|^{2})d\mu\\
&\leq&\underset{k\rightarrow+\infty}\liminf \int_{V}(|\Delta u_{\lambda_{k}}|^{2}+|\nabla u_{\lambda_{k}}|^{2}+(\lambda_{k}a+1)|u_{\lambda_{k}}|^{2})d\mu\\
&=&\underset{k\rightarrow+\infty}\liminf\int_{V}|u_{\lambda_{k}}|^{p}d\mu\\
&=&\int_{V}|u_{0}|^{p}d\mu.
\end{eqnarray*}
Noticing that $u_{0}|_{\Omega^{c}}=0$, we get
$$\int_{\Omega\cup\partial\Omega}(|\Delta u_{0}|^{2}+|\nabla u_{0}|^{2})d\mu+\int_{\Omega}|u_{0}|^{2}d\mu\leq\int_{\Omega}|u_{0}|^{p}d\mu.$$
Then there exists $t\in(0,1]$ such that $tu_{0}\in\mathcal{N}_{\Omega}$, i.e.
$$\int_{\Omega\cup\partial\Omega}(|t\Delta u_{0}|^{2}+|t\nabla u_{0}|^{2})d\mu+\int_{\Omega}|tu_{0}|^{2}d\mu=\int_{\Omega}|tu_{0}|^{p}d\mu.$$
This implies that
\begin{eqnarray*}
J_{\Omega}(tu_{0})&=&
\frac{p-2}{2p}\int_{\Omega\cup\partial\Omega}(|t\Delta u_{0}|^{2}+|t\nabla u_{0}|^{2})d\mu+\int_{\Omega}|tu_{0}|^{2}d\mu\\
&\leq&\frac{p-2}{2p}\int_{V}(|t\Delta u_{0}|^{2}+|t\nabla u_{0}|^{2}+|tu_{0}|^{2})d\mu\\
&\leq&\underset{k\rightarrow\infty}\liminf\left[\frac{p-2}{2p}\int_{V}(|t\Delta u_{\lambda_{k}}|^{2}+|t\nabla u_{\lambda_{k}}|^{2}+(\lambda_{k}a+1)|tu_{\lambda_{k}}|^{2})d\mu\right]\\
&=&\underset{k\rightarrow\infty}\liminf J_{\lambda_{k}}(tu_{\lambda_{k}})\\
&\leq&\underset{k\rightarrow\infty}\liminf J_{\lambda_{k}}(u_{\lambda_{k}})=M.
\end{eqnarray*}
Consequently, $M\geq m_{\Omega}$. Then we get that
$$\underset{\lambda\rightarrow\infty}\lim m_{\lambda}=m_{\Omega}.$$
\end{proof}

\noindent{\bf The proof of Theorem 1.2 and 1.3.} We need to prove that for any sequence $\lambda_{k}\rightarrow\infty$, the corresponding $u_{\lambda_{k}}\in\mathcal{N}_{\lambda_{k}}$ satisfying $J_{\lambda_{k}}(u_{\lambda_{k}})=m_{\lambda_{k}}$ converges in $W^{2,2}(V)$ to a ground state solution $u_{0}$ of
\eqref{dirichlet} up to a subsequence.

By Remark \ref{psbound}, we have that $u_{\lambda_{k}}$ is bounded in $E_{\lambda_{k}}$ and the upper-bound is independent of $\lambda_{k}$. Consequently, $\{u_{\lambda_{k}}\}$ is also bounded in $W^{2,2}(V)$. Therefore, we can assume that there exists some $u_0\in W^{2,2}(V)\bigcap L^q(V)$ for any $q\in [2,+\infty)$, such that
$$u_{\lambda_{k}}\rightarrow u_{0}\ \ \hbox{in}\ \ L^{q}(V)$$
and
$$u_{\lambda_{k}}\rightharpoonup u_{0}\ \ \hbox{in}\ \ W^{2,2}(V).$$
From Lemma \ref{nu}, we have that $u_{0}\not\equiv0$ and on the other hand, as what we have done in Lemma \ref{groundstate}, we have that $u_{0}|_{\Omega^{c}}=0$. First, we claim that as $k\rightarrow\infty$, there hold
\begin{equation}\label{aterm}
\lambda_{k}\int_{V}a|u_{\lambda_{k}}|^{2}d\mu\rightarrow 0
\end{equation}
and
\begin{equation}\label{wterm}
\int_{V}|\Delta u_{\lambda_{k}}|^{2}+|\nabla u_{\lambda_{k}}|^{2}d\mu\rightarrow\int_{V}|\Delta u_{0}|^{2}+|\nabla u_{0}|^{2}d\mu.
\end{equation}
Otherwise, if for some $\theta>0$, there holds
$$\underset{k\rightarrow\infty}\lim\lambda_{k}\int_{V}a|u_{\lambda_{k}}|^{2}d\mu=\theta,$$
it follows that
\begin{eqnarray*}
\int_{\Omega\cup\partial\Omega}|\Delta u_{0}|^{2}+|\nabla u_{0}|^{2}d\mu+\int_{\Omega}|u_{0}|^{2}d\mu
&<&\int_{V}(|\Delta u_{0}|^{2}+|\nabla u_{0}|^{2}+|u_{0}|^{2})d\mu+\theta\\
&\leq&\underset{k\rightarrow+\infty}\liminf \int_{V}(|\Delta u_{\lambda_{k}}|^{2}+|\nabla u_{\lambda_{k}}|^{2}+(\lambda_{k}a+1)|u_{\lambda_{k}}|^{2})d\mu\\
&=&\underset{k\rightarrow+\infty}\liminf\int_{V}|u_{\lambda_{k}}|^{p}d\mu\\
&=&\int_{\Omega}|u_{0}|^{p}d\mu.
\end{eqnarray*}
By $p>2$, there exists some $t\in(0,1)$ such that $tu_{0}\in\mathcal{N}_{\Omega}$. On the other hand, if
$$\underset{k\rightarrow+\infty}\liminf\int_{V}|\Delta u_{\lambda_{k}}|^{2}+|\nabla u_{\lambda_{k}}|^{2}d\mu>\int_{V}|\Delta u_{0}|^{2}+|\nabla u_{0}|^{2}d\mu,$$
we also have $\int_{\Omega\cup\partial\Omega}|\Delta u_{0}|^{2}+|\nabla u_{0}|^{2}d\mu+\int_{\Omega}|u_{0}|^{2}d\mu<\int_{\Omega}|u_{0}|^{p}d\mu$. Then in both cases, we can find $t\in(0,1)$ such that $tu_{0}\in\mathcal{N}_{\Omega}$. Consequently, we obtain that
\begin{eqnarray*}
J_{\Omega}(tu_{0})&=&
\frac{p-2}{2p}\left(\int_{\Omega\cup \partial\Omega}|t\Delta u_{0}|^{2}+|t\nabla u_{0}|^{2}d\mu+\int_{\Omega}|tu_{0}|^{2}d\mu\right)\\
&=&\frac{p-2}{2p}t^{2}\left(\int_{\Omega\cup \partial\Omega}|\Delta u_{0}|^{2}+|\nabla u_{0}|^{2}d\mu+\int_{\Omega}|u_{0}|^{2}d\mu\right)\\
&<&\frac{p-2}{2p}\int_{V}(|\Delta u_{0}|^{2}+|\nabla u_{0}|^{2}+|u_{0}|^{2})d\mu\\
&\leq&\underset{k\rightarrow+\infty}\liminf\left[\frac{p-2}{2p}\int_{V}(|\Delta u_{\lambda_{k}}|^{2}+|\nabla u_{\lambda_{k}}|^{2}+(\lambda_{k}a+1)|u_{\lambda_{k}}|^{2})d\mu\right]\\
&=&\underset{k\rightarrow+\infty}\liminf J_{\lambda_{k}}(u_{\lambda_{k}})\\
&=&m_{\Omega},
\end{eqnarray*}
which is contradiction and the claim is proved.

Now we can prove that $u_{0}$ is a ground state solution of \eqref{dirichlet}.
In fact, since $J^{'}_{\lambda_{k}}(u_{\lambda_{k}})=0$, for any $\phi\in H(\Omega)\subset W^{2,2}(V)$, we have
$$\int_{V}(\Delta u_{\lambda_{k}}\Delta\phi+\nabla u_{\lambda_{k}}\nabla\phi+(\lambda_{k}a+1)u_{\lambda_{k}}\phi)d\mu=\int_{V}|u_{\lambda_{k}}|^{p-2}u_{\lambda_{k}}\phi d\mu.$$
Since $a(x)\phi(x)\equiv 0$, for any $x\in V$, we get
$$\int_{\Omega\cup\partial\Omega}(\Delta u_{\lambda_{k}}\Delta\phi+\nabla u_{\lambda_{k}}\nabla\phi)d\mu+\int_{\Omega}u_{\lambda_{k}}\phi d\mu=\int_{\Omega}|u_{\lambda_{k}}|^{p-2}u_{\lambda_{k}}\phi d\mu.$$
Let $k\rightarrow \infty$, the above equality becomes
$$\int_{\Omega\cup\partial\Omega}(\Delta u_{0}\Delta\phi+\nabla u_{0}\nabla\phi)d\mu+\int_{\Omega}u_{0}\phi d\mu=\int_{\Omega}|u_{0}|^{p-2}u_{0}\phi d\mu,$$
which tells us that $J^{'}_{\Omega}(u_{0})=0$, $u_0\in \mathcal{N}_{\Omega}$ and $u_0$ is a solution of \eqref{dirichlet}.

On the other hand, by \eqref{aterm} and \eqref{wterm}, we have
\begin{eqnarray*}
J_{\lambda_{k}}(u_{\lambda_{k}})
&=&\frac{1}{2}\int_{V}|\Delta u_{\lambda_{k}}|^{2}+|\nabla u_{\lambda_{k}}|^{2}+(\lambda_{k}a+1)|u_{\lambda_{k}}|^{2}d\mu-\frac{1}{p}\int_{V}|u_{\lambda_{k}}|^{p}d\mu\\
&=&\frac{1}{2}\int_{V}|\Delta u_{0}|^{2}+|\nabla u_{0}|^{2}+|u_{0}|^{2}d\mu-\frac{1}{p}\int_{V}|u_{0}|^{p}d\mu+o_{k}(1)\\
&=&\frac{1}{2}\int_{\Omega\cup\partial\Omega}|\Delta u_{0}|^{2}+|\nabla u_{0}|^{2}d\mu+\int_{\Omega}|u_{0}|^{2}d\mu-\frac{1}{p}\int_{\Omega}|u_{0}|^{p}d\mu+o_{k}(1)\\
&=&J_{\Omega}(u_{0})+o_{k}(1).
\end{eqnarray*}
Since $J_{\lambda_{k}}(u_{\lambda_{k}})=m_{\lambda_{k}}$, Lemma \ref{groundstate} tells $J_{\Omega}(u_{0})=m_{\Omega}$. Thus we get that $u_{0}$ is a solution of \eqref{dirichlet} which achieves the ground state and both Theorem \ref{existence2} and Theorem \ref{convergence} are proved.
$\Box$

\begin{remark}\label{strongconvergence}
Furthermore, we have $\underset{k\rightarrow \infty}\lim\|u_{\lambda_{k}}-u_{0}\|_{E_{\lambda_{k}}}= 0$ and consequently, $\{u_{\lambda_k}\}$ converges strongly to $u_0$ in $W^{2,2}(\Omega)$.
\end{remark}

\begin{proof}
Indeed, since $u_{\lambda_{k}}\in\mathcal{N}_{\lambda_{k}}$ and $u_{0}|_{\Omega^{c}}=0$, we have
\begin{eqnarray*}
\|u_{\lambda_{k}}-u_{0}\|_{E_{\lambda_{k}}}^2
&=&\int_{V}|\Delta (u_{\lambda_{k}}-u_{0})|^{2}+|\nabla (u_{\lambda_{k}}-u_{0})|^{2}+(\lambda_{k}a+1)|u_{\lambda_{k}}-u_{0}|^{2}d\mu\\
&=&\|u_{\lambda_{k}}\|_{E_{\lambda_{k}}}^2+\|u_{0}\|_{E_{\lambda_{k}}}^2-2\int_{V}\Delta u_{\lambda_{k}}\Delta u_{0}+\nabla u_{\lambda_{k}}\nabla u_{0}+(\lambda_{k}a+1)u_{\lambda_{k}}u_{0}d\mu\\
&=&\|u_{\lambda_{k}}\|_{E_{\lambda_{k}}}^2+\|u_{0}\|_{H(\Omega)}^2-2\int_{\Omega\cup\partial\Omega}\Delta u_{\lambda_{k}}\Delta u_{0}+\nabla u_{\lambda_{k}}\nabla u_{0}d\mu+\int_{\Omega}u_{\lambda_{k}}u_{0}d\mu\\
&=&\|u_{\lambda_{k}}\|_{E_{\lambda_{k}}}^2+\|u_{0}\|_{H(\Omega)}^2
-2\|u_{0}\|_{H(\Omega)}^2+o_{k}(1)\\
&=&\|u_{\lambda_{k}}\|_{E_{\lambda_{k}}}^2-\|u_{0}\|_{H(\Omega)}^2+o_{k}(1)\\
&=&\int_{V}|u_{\lambda_{k}}|^{p}d\mu-\int_{\Omega}|u_{0}|^{p}d\mu+o_{k}(1)\\
&=&\int_{V}|u_{\lambda_{k}}|^{p}d\mu-\int_{V}|u_{0}|^{p}d\mu+o_{k}(1)\\
&=&o_k(1).
\end{eqnarray*}
\end{proof}

\end{document}